\documentclass[11pt]{amsart}
\usepackage{amscd}
\usepackage{amsmath,amssymb,latexsym}
\usepackage{pb-diagram}
\usepackage{enumerate}
\usepackage{amsfonts}
\usepackage{amsthm}

\newtheorem{Thm}[equation]{Theorem}

\newtheorem{Prop}[equation]{Proposition}
\newtheorem{Lem}[equation]{Lemma}
\newtheorem{Remark}[equation]{Remark}
\newtheorem{Def}[equation]{Definition}
\newcommand{\ra}{\rangle}
\newcommand{\la}{\langle}
\newcommand{\ti}{\tilde}
\newcommand{\C}{\mathbb C}

\newcommand{\mR}{\mathcal R}

\newcommand{\bR}{\mathbb R}

\newcommand{\bB}{\mathbb B}

\newcommand{\bH}{\mathbb H}
\newcommand{\bW}{\mathbb W}
\newcommand{\mW}{\mathcal W}
\newcommand{\mB}{\mathcal B}

\newcommand{\mS}{\mathcal S}

\newcommand{\mL}{\mathcal L}

\newcommand{\mF}{\mathcal F}
\newcommand{\mK}{\mathcal K}
\newcommand{\mO}{\mathcal O}
\newcommand{\G}{\mathbb G}
\newcommand{\mg}{\mathfrak g}
\newcommand{\st}{\sqrt\tau}
\newcommand\T{\operatorname{Tr}}
\newcommand\N{\operatorname{Nm}}
\newcommand\Ker{\operatorname{Ker}}
\newcommand\End{\operatorname{End}}
\newcommand\Aut{\operatorname{Aut}}
\newcommand\im{\operatorname{Im}}
\newcommand\Res{\operatorname{Res}}

\newcommand\Id{\operatorname{Id}}

\newcommand\ind{\operatorname{ind}}
\newcommand\Ind{\operatorname{Ind}}

\theoremstyle{plain}
\title[Fourier transforms on the basic affine space]
      {Fourier transforms on the basic affine space of a quasi-split group}
\author{Nadya Gurevich and David Kazhdan}
\address{School of Mathematics,
Ben Gurion University of the Negev, POB 653, Be'er Sheva 84105, Israel}
\address{Einstein Institute of Mathematics,
  The Hebrew University of Jerusalem, Givat Ram, Jerusalem, 9190401, Israel}
\email{ngur@math.bgu.ac.il}
\email{kazhdan@math.huji.ac.il}
\numberwithin{equation}{section}

\begin{document}
\maketitle

\begin{abstract} 
We extend the Gelfand and Graev construction of generalized
Fourier transforms on basic affine space from split groups
to quasi-split groups over a  local  non-archimedean field $F$.

\end{abstract}  
\section{Introduction}\label{sec:intro}

\subsubsection{Notation}
\begin{itemize}
\item Let $F$ be  a local non-archimedean
  field with the norm $|\cdot|=|\cdot|_F$,
  the ring of integers $\mO$ and a fixed uniformizer
  $\varpi$ such that $|\varpi|=q^{-1},$
  where $q$ is the cardinality of the residue field.
\item We fix a non-trivial additive character $\psi$ throughout the paper.
  The self-dual Haar measure $dx$ on $F$ with respect to $\psi$ defines the
  Haar measure $d^\times x=\frac{dx}{|x|}$ on $F^\times.$

\item  For a quadratic extension $K$ of $F$ we denote by $\chi_K$
  the quadratic character of $F^\times,$
  associated to $K$  by class field theory. We also denote by $\chi_0$
  the trivial character of $F^\times$.
  
\item For a space $Y$ over $F$ we denote by $\mS^\infty(Y)$ (resp.
  $\mS_c(Y)$) the space of locally constant  (resp. locally
  constant of compact support) functions  on $Y$.
\item Throughout this paper we use boldface characters for group
  schemes over $F$, such as ${\bf H}$, and plain text characters
  for their group of $F$-points, such as $H$. 
\item
Let ${\bf G}$ be a  simply-connected quasi-split group defined over $F$
with a maximal $F$-split torus ${\bf T'}$ and the maximal torus
${\bf T}={\bf Z_G(T')}$.  We fix a Borel subgroup ${\bf B}$ of ${\bf G}$
containing ${\bf T}$ so that ${\bf B}={\bf T}\cdot {\bf U}$.
We write ${\bf U}^{op}$ for the unipotent radical of the opposite
Borel subgroup. 
\item
The Weyl group $W=N_G(T')/T$ acts on $T$ by conjugation and
we write $t^w$ for $w^{-1}tw$ for all $t\in T$, $w\in W$.
\item
The quotient  $X=U\backslash G$ is called the basic affine space
of $G$. For any $g\in G$ we write $[g]$ for the element $Ug$ in $X$. 
The space $X$ admits unique,
up to a scalar, $G$-invariant measure $\omega_X$.
The precise choice of $\omega_X$  is not important for general $G$,
but will be fixed for groups of rank $1$. 
\end{itemize}

\subsection{Fourier transforms on the
  basic affine space of a quasi-split group}\label{intro:part2}

We define a unitary representation $\theta $ of 
the group $G\times T$ on  $L^2(X,\omega_X)$ by:  
 $$\theta (g,t)f([h])=\delta^{1/2}_B(t)f([t^{-1}hg]),$$
where $\delta_B$ is the modular character.

For split groups Gelfand and Graev in
\cite{GelfandGraev},
see also \cite{KazhdanLaumon},
\cite{Kazhdan}, extended the action $\theta$ of $G\times T$
to a representation of $G\times (T\rtimes W),$ so that
every element $w$ of $W$
acts on $L^2(X,\omega_X)$ by an operator $\Phi_w$, called
a generalized Fourier transform.
Our paper has two goals:
\begin{itemize}
\item To extend the construction by Gelfand and Graev to
  quasi-split groups.
\item To show that  the Whittaker map intertwines  the action
  of $W$ on a dense subspace $\mS_0(X)$ in $L^2(X)$
with the natural action of $W$ on the space of Whittaker vectors.
We show (see Theorem \ref{main}) that this property
characterizes uniquely the operators $\Phi_w$.
\end{itemize}
\subsubsection{Whittaker map}
Fix a non-degenerate character $\Psi$ of $U^{op}$.
The map
\begin{equation}\label{def:W:map}
\mW_\Psi:\mS_c(X)\rightarrow \mS_c(T), \quad
\mW_\Psi(f)(t)=\int\limits_{U^{op}} \theta(t) f([u]) \Psi^{-1}(u) du,
\end{equation}
 defines an isomorphism $\mS_c(X)_{U^{op},\Psi}\simeq \mS_c(T)$. 

We define an action of $W$ on $\mS_c(T)$.
For split groups  set
$$w\cdot \varphi(t)=\varphi(t^w).$$
For quasi-split groups see  Definition \ref{qs:action}. 

We define (see \ref{def:S0:general}) a $G\times T$ submodule $\mS_0(X)$
that is dense in $L^2(X)$ and put
$$\mS_0(T)=\mW_\Psi(\mS_0(X))\simeq \mS_0(X)_{U^{op},\Psi}.$$

There is a natural map
$\kappa_{\Psi}:\End_G(\mS_0(X))\rightarrow \End_{\C}(\mS_0(X)_{U^{op},\Psi})=
\End_{\C}(\mS_0(T))$
such that for every $\mB\in \End_G(\mS_0(X))$
the following diagram is commutative.
$$
\begin{diagram}
\node{\mS_0(X)}
\arrow{e,t}{\mB}
\arrow{s,l}{\mW_\Psi}
\node {\mS_0(X)}
\arrow{s,r}{\mW_\Psi} \\
\node{\mS_0(T)}
\arrow{e,t}{\kappa_\Psi(\mB)} 
\node{\mS_0(T)}
\end{diagram}
$$

We prove in Proposition \ref{injective:general} that 
the map $\kappa_\Psi$ is injective.

\subsubsection{Main Theorem}
With this notation  we formulate our main result.  
\begin{Thm}
  \label{main}
  There exists unique family of unitary operators
  $\Phi_w, w\in W$, on $L^2(X,\omega_X)$, preserving the space
  $\mS_0(X)$ and  satisfying:
  \begin{equation}
    \label{request}
    \left\{
    \begin{array}  {ll}
      \Phi_w\circ \theta(g,t^w)=
      \theta(g, t)\circ \Phi_w &  \forall w\in W, t\in T, g\in G\\
      \Phi_{w_1}\Phi_{w_2}=\Phi_{w_1w_2} &  \forall w_1, w_2 \in W\\
      \kappa_\Psi(\Phi_w)(\varphi)=w\cdot \varphi &
      \forall w\in W, \varphi\in \mS_0(T)
    \end{array}
    \right.
\end{equation}
\end{Thm}

Let us sketch the proof. 
\begin{enumerate}
\item
  First consider a quasi-split, almost simple,
  simply-connected group $G_1$ of rank one. The group $G_1$ is
  isomorphic to either $\Res_LSL_2$ or $\Res_LSU_3$
  for a finite extension $L$ of $F$.
  Without loss of generality  we can assume that $L=F$.
  In both cases the Weyl group $W=\{e,s\}$ consists of two elements.
  We shall define the generalized
  Fourier operator $\Phi_s$, separately for these two cases.

\begin{itemize}\item
In the  case $G_1=SL_2$ the set $X$ can be  identified with
$V-0$ for a symplectic two dimensional plane $V$. In this case
$\Phi_s\in \Aut(L^2(X))=\Aut(L^2(V))$ is defined to be the
classical Fourier transform with respect to the symplectic form on $V$.
Theorem \ref{main} in this case is proven in Section \ref{sec:SL2}.
\item In the case  $G_1=SU_3$, the set $X$ can be  identified
  with the set of non-zero isotropic vectors in a $6$
  dimensional quadratic space.
  The treatment of this case is the crux of the paper. 
In \cite{GurevichKazhdan} we have  defined a unitary operator
$\Phi\in L^2(X)$ of order $2$, commuting with $G_1$ and
anti-commuting with $T'$, and provided an explicit formula
for the  restriction of $\Phi$ to the space
$\mS_c(X)$. We put $\Phi_s=\Phi$
and prove Theorem  \ref{main} in this case  in Section \ref{sec:SU3}.
\end{itemize}
\item
For a general quasi-split group  $G$ and  any simple reflection
$s$ we, using the results for groups of rank $1$, define a unitary involution
$\Phi_{s}\in \Aut(L^2(X)),$ satisfying
 $$   \left\{
    \begin{array}  {ll}
      \Phi_{s}\circ \theta(g,t^s)=\theta(g, t)\circ \Phi_{s} &
      \forall  t\in T, g\in G\\
      \kappa_\Psi(\Phi_s)(\varphi)=s\cdot \varphi & \forall \varphi\in \mS_0(T) 
    \end{array}
    \right..$$
\item
For arbitrary $w\in W$ with a presentation
$w=s_{1}\cdot s_{2}\cdot \ldots \cdot s_{n}$
as a product of simple reflections 
we define
$\Phi_w=\Phi_{s_1}\circ \Phi_{s_2} \ldots \circ \Phi_{s_n}.$
Hence the operators $\Phi_w$ are unitary and possess
the desired equivariance properties. It remains to prove
that $\Phi_w$ does not depend on the presentation. 
For every $\varphi\in \mS_0(T)$ one has
$\kappa_\Psi(\Phi_w)(\varphi)=w\cdot \varphi$ and so $\kappa_\Psi(\Phi_w)$
does not depend  on the presentation of $w$. Since $\kappa_\Psi$ is injective,
the operator $\Phi_w$ does not depend on the presentation of $w$ as well.
In particular, $\Phi_{w_1}\circ \Phi_{w_2}=\Phi_{w_1w_2}$ for $w_1,w_2\in W$
and the operators $\{\Phi_w, w\in W\}$ satisfy \ref{request}. 
\end{enumerate}

\begin{Remark} We expect that a similar strategy can be applied
  to prove Theorem \ref{main} for $F=\bR$.
\end{Remark}
  
\subsection*{Acknowledgment}
The research of the second author is
partially supported by the ERC grant No $669655$.
We thank the referee for careful reading of the paper
and pointing several inaccuracies in the first version. 

\section{On the space $\mS_0(X)$}\label{sec:S0}

In \cite{BravermanKazhdan} the authors have defined for split groups
the spaces $$\mS(X)=\sum_{w\in W} \Phi_w(\mS_c(X)), \quad \mS^0(X)=\cap_{w\in W} \Phi_w(\mS_c(X)).$$
In particular 
$$\mS^0(X)\subset \mS_c(X)\subset \mS(X)\subset L^2(X,\omega_X)$$
and the  spaces $\mS^0(X), \mS(X)$ are preserved by the family of
operators ${\Phi_w,w \in W}$.
The space $\mS(X)$, called  Schwartz space, is potentially
important for  construction of integral representations of $L$-functions.

The description of the Schwartz space $\mS(X)$ explicitly is a deep problem.
For example for $G=SL_2$ one has $\mS(X)=\mS_c(V)$ and
for $G=SU_3$ the space $\mS(X)$ can be identified with
the space of smooth vectors in the unitary minimal representation
of a group $SO(8)$ containing $SU_3$ inside its Levi subgroup
$GL_1\times SO(6),$ see \cite{GurevichKazhdan}.

The space $\mS_0(X)$ in this paper is contained in  $\mS^0(X)$.
Let us highlight its useful properties:
\begin{itemize}
  \item It is explicitly given as an intersection of
    kernels of certain partial Mellin transforms.
  \item The Fourier transforms corresponding to simple reflections
    preserve this space and 
    can be written as integral operators with explicitly
    given continuous kernels.
  \item The family $\{\Phi_w\}$ is unique for given $\mS_0(X).$
\end{itemize}

  On the other hand this space is  not canonical and can easily be replaced
by other subspaces in $\mS_c(X)$, dense in $L^2(X,\omega_X)$
and preserved by $\Phi_w$, for example by $\mS^0(X).$ 

The space $\mS_0(X)$ will be defined separately for the
groups of rank one, and, based on this, for general group.

The density of $\mS_0(X)$ in $L^2(X,\omega_X)$ is the consequence
of Proposition \ref{denseness} below.

Consider a finite set
$$\bB=\{(L_i,a_i , \chi_i),\quad 1\le i \le k\},$$
  where $L_i$ is a finite extension of $F$,
  $a_i:L_i^\times \hookrightarrow T$ is an embedding and 
  $\chi_i$ is a character of $L_i^\times$. 
  For each $(L_i,a_i,\chi_i)$ consider a  partial Mellin transform
  $P_i:\mS_c(X)\rightarrow \mS^\infty(X)$ defined by
  $$P_i(f)=\int\limits_{L_i^\times} \theta(a_i(y))f \,\,\chi_i(y) d^\times y.$$

  Define $\mS_{\bB}(X)=\cap_{i=1}^k \Ker(P_i).$ It is a
  $G\times T$ invariant subspace of $\mS_c(X)$.
  
  The following proposition will be repeatedly used in the paper.
  \begin{Prop}\label{denseness}
The space $\mS_\bB(X)$ is dense in $L^2(X,\omega_X)$. 
\end{Prop}

  \begin{proof}
    Let us prove this first for the case all the characters $\chi_i$
    are not unitary. Precisely assume that all $\chi_i$ satisfy
$|\chi_i|=|\cdot|^{b_i}$ with   real $b_i\neq 0$ for all $i$.  
Let $b=\min(|b_i|)>0$. 

To show that the space $\mS_\bB(X)$ is dense, assume
existence of a non-zero function $f\in \mS_\bB(X)^\perp\subset L^2(X,\omega_X).$
Since $\mS_c(X)$ is dense in $L^2(X)$
there exists  a function $g\in \mS_c(X)$ such that $\la f, g\ra\neq 0$.
      
  Denote by $\varpi_i$ an uniformizer of $L_i$. 
  For any $n\in \mathbb N$  define operators $E^i_n, E_n$   on $\mS_c(X)$ by
$E_n=\Pi_{i=1}^k E^i_n,$ where
  $$E^i_n=\left\{
  \begin{array}{ll} \Id-\theta(a_i(\varpi_i)^n)\chi_i(\varpi_i^n)& b_i>0\\
    \Id-\theta(a_i(\varpi_i)^{-n})\chi_i(\varpi_i^{-n})& b_i<0\end{array}.
    \right. 
$$
    Clearly, $E_n(g)\in \mS_B(X)$. Set $g_n=g-E_n(g)$.
Note that  $|\chi_i(\varpi_i)|$ (resp. $|\chi_i(\varpi_i^{-1})|$)
  is bounded by $q^{-b}$ for $b_i>0$ (resp. $b_i<0$) for any $i$.
 Moreover the  action $\theta(a_i(\varpi_i))$ is unitary.
 This implies  $\|g_n\|\le q^{-nb}(2^k-1)\|g\|.$
Hence
$$0\neq |\la g,f\ra|=|\la g_n,f\ra|\le q^{-nb}(2^k-1)
\|g\|\cdot \|f\|\to 0$$
as $n\to \infty$, which is a contradiction.

Now let us treat the general set of characters $\bB$.
For any compact subset $\mK$ in $X$ let $\mS_\bB(X;\mK)$ be
the space of functions in $\mS_\bB(X)$ supported on $\mK$. Obviously,
    $\mS_\bB(X)=\cup_\mK \mS_\bB(X;\mK)$. 

Since the action of $T$ on $X$ is free, for any character
 $\chi$ of $T$ there exists a  smooth function
$h$ on $X$, such that
$$h([t^{-1}g])=\chi(t)h([g]), \quad h([g])\neq 0, \quad \forall [g]\in X,
t\in T.$$
Multiplication on $h$ defines a $T$-equivariant isomorphism between
    $\mS_\bB(X)$ and $\mS_{\bB'}(X),$ where
$\bB'=\{(L_i, a_i, \chi_i\cdot \chi\circ a_i)\},$
which is also homeomorphism  between $\mS_\bB(X;\mK)$ and
$\mS_{\bB'}(X;\mK)$ for all compact $\mK\subset X$.
Hence $\mS_\bB(X)$ is dense if and only if $\mS_{\bB'}(X )$ is dense in
$L^2(X, \omega_X)$.
By choosing appropriate $\chi$ we can ensure that $\bB'$ does not contain
unitary characters. We are done. 

\end{proof}

\section{$G_1=SL_2$}\label{sec:SL2}

Let $(V, \la\cdot,\cdot\ra_V)$ be a two dimensional symplectic space
with the standard basis $e_1,e_2$ such that $\la e_1,e_2\ra_V=1$.

The group $G_1$ acts on $V$ on the right, preserving the symplectic
form. Let $B_1=T_1\cdot U_1$
be the Borel group, stabilizing the line $Fe_2$. The 
space $X=U_1\backslash G_1$ is identified with
$V-0$ via $[g]\mapsto e_2g$.
The $G_1$-invariant measure $\omega_X$ on $X$ is fixed to be
the self-dual measure $|dv|$ on $V$ with respect to
the additive character $\psi$ and the symplectic form on $V$.

The Fourier transform $\Phi\in \Aut(\mS_c(V))$ is defined
by the formula
$$\Phi(f)(w)=\int\limits_V f(v)\psi(\la w,v\ra_V) dv.$$

The following properties of $\Phi$ are well-known:

\begin{Prop}\label{SL2:Phi:properties}
  \begin{enumerate}
    \item
 $\Phi$ extends to a unitary involution on $L^2(V, |dv|)=L^2(X,\omega_X)$
\item
  $\theta(t,g)\circ \Phi=
  \Phi\circ \theta(t^{-1},g)$ for all $(t,g)\in T_1\times G_1.$
  \end{enumerate}
  \end{Prop}

For a function $f$ on $X$ the argument will be denoted either
as a class $[g]$ or as a vector  $(x,y)=xe_1+ye_2\in V-0$. 

We define certain typical elements of $G_1:$
$$x(r)=\left(\begin{array}{cc} 1& 0 \\r & 1 \end{array}\right),
\,\,
\,\,
t(a)=\left(\begin{array}{cc} a& 0 \\0 & a^{-1} \end{array}\right),
\,\,
n_s=\left(\begin{array}{cc} 0& -1 \\1 & 0 \end{array}\right).
$$

One has $\alpha(t(a))=a^2$ for the unique positive root $\alpha$
of $G_1$ with respect to $T_1$.

\subsection{The space $\mS_0(X)$}\label{def:S0:SL2}
Let $\bB$ be the set of  two triples
$$\bB=\{(F^\times, t:F^\times \rightarrow T_1, \chi_{\pm}(y)=|y|^{\pm 1})\}.$$
We define $\mS_0(X)$ to be $\mS_{\bB}(X),$ see section \ref{sec:S0}
for the definition.
It is obviously a $G_1\times T_1$ representation and is dense in 
in $L^2(X,\omega_X)$ by Proposition \ref{denseness}.

\begin{Prop}\label{SL2:S0} 
  The operator $\Phi$ preserves $\mS_0(X)$. 
  \end{Prop}

\begin{proof}
  First note, that for any $f\in \mS_0(X),$
  the function  $\Phi(f)$ belongs to $\mS_c(X)$. Indeed,
  the germ $[\Phi(f)]_0$ of $\Phi(f)$  at zero is
constant and equals 
   $$[\Phi(f)]_0=\int\limits_V f(v) dv=
 \int\limits_F \int\limits_{F^\times} \theta(t(x))f(1,y)|x| d^\times x dy=
 \int\limits_F P(\chi_+)(f)(1,y) dy=0.$$

For any character $\chi$ of $T_1$ one has
 $$P(\chi)(\Phi(f))(v)=
 \int\limits_{T_1} \theta(t)\Phi(f)(v)\chi(t) dt=
 \int\limits_{T_1} \Phi(\theta(t^{-1})f)(v)\chi(t) dt.$$

 Since $f$ is of compact support, the integral defining
 $\Phi(f)$ is taken over a compact set in $X$,
 and hence the integral over $T_1$ can also be replaced by an
 integral over a compact set. By interchanging the order of integration
 we see that if $f\in \Ker P(\chi^{-1})$ then
 $\Phi(f)\in \Ker P(\chi)$. 
 
Hence for $f\in \mS_0(X)$ the function
$\Phi(f)$ belongs to $\mS_0(X).$ This proves the Lemma.
 
\end{proof}

\subsubsection{The Whittaker map}
We fix a character  $\Psi$ on $U^{op}_1$  by ${\Psi(x(r))=\psi(r)}$.
The Whittaker map $\mW_\Psi:\mS_c(X)\rightarrow \mS_c(T_1)$ is defined by
$$\mW_\Psi(f)(t)=\int\limits_{U^{op}_1} \theta(t)f([u]) \Psi^{-1}(u) du.$$

The map $\mW_\Psi$ defines an isomorphism
$\mS_0(X)_{U^{op}_1,\Psi}\simeq \mS_0(T_1),$
where  $\mS_0(T_1)=\mW_\Psi(\mS_0(X))$, which induces the  map 
$$\kappa_\Psi:\End_{G_1}(\mS_0(X))\rightarrow
\End_\C(\mS_0(X)_{U^{op},\Psi})=\End_\C(\mS_0(T_1)).$$

\begin{Lem}\label{SL2:kappa:injectivity}
  $\kappa_\Psi$ is injective.
\end{Lem}
\begin{proof} See the proof of \ref{injective:general}
  for a general quasi-split $G$. 
\end{proof}

\begin{Def}\label{SL2:W:action}
  We define an action of $W$ on $S_c(T_1)$ by
  $$s\cdot \varphi(t)=\varphi(t^s), \quad \varphi\in \mS_c(T_1).$$
\end{Def}

\begin{Prop}\label{SL2:kappa}
  For any $\varphi\in \mS_0(T_1)$ one has
  $\kappa_\Psi(\Phi)(\varphi)=s\cdot \varphi.$
\end{Prop}

\begin{proof}
Any function in $\mS_0(T_1)$ is of the form $\mW_\Psi(f)$ for $f\in \mS_0(X)$.
It is enough to show that
\begin{equation}
  \mW_\Psi(\Phi(f))(1)=\mW_\Psi(f)(1).
  \end{equation}
Indeed, once this is proven one has for $t\in T_1$
$$\mW_\Psi(\Phi(f))(t)=\mW_\Psi(\theta(t)\Phi(f))(1)=
\mW_\Psi(\Phi(\theta(t^s)f)(1)=$$
$$\mW_\Psi(\theta(t^s)(f))(1)=\mW_\Psi(f)(t^s).$$

There is a injective map with open dense image
$$j: T_1\times U^{op}_1\rightarrow X, \quad j(t,u)= [t^{-1} u]$$
and the push-forward of the measure $\delta_B(t) d t\, d u$ on
$T_1\times U_1^{op}$
equals $dv$.

\begin{equation}\label{SL2:Phi}
\Phi(f)([g])=
\int\limits_{U^{op}_1}
\int\limits_{T_1}f([t^{-1}  u ])
\psi(\la [g], [t^{-1}   u]\ra_V) \delta_B(t) d t du.
\end{equation}
Hence
  $$\mW_\Psi(\Phi(f))(1)=
  \int\limits_{U^{op}_1} \Phi(f)([u]) \Psi(u)^{-1}du=$$
$$\int\limits_{U^{op}_1} \int\limits_{U^{op}_1} \int\limits_{T_1}
  f([t^{-1}u']) \psi(\la [u], [t^{-1}u']\ra_V)
  \Psi(u^{-1})   \delta_B(t) dt\, du'\,du=$$
$$\int\limits_{U^{op}_1}  \int\limits_{T_1}\left(\int\limits_{U^{op}_1}
f([t^{-1} u'])\Psi(u'^{-1}) du'\right)
\psi(\la [1],[t^{-1} u]\ra_V ) \Psi(u)  \delta_B(t) dt\, du= $$
$$\int\limits_{U^{op}_1}  \int\limits_{T_1} \mW_\Psi(f)(t)
\psi(\la [1],[t^{-1} u]\ra_V) \Psi(u)  \delta^{1/2}_B(t) dt\, du.$$

Put $t=t(b)$ and $u=x(r)$ and notice that
$\la [1],[t^{-1}u]\ra_V=-br$. Then

$$\mW_\Psi(\Phi(f))(1)=
\int\limits_{F}
\left(\int\limits_{F}\mW_\Psi(f)(t(b))\psi(-br) db\right)  \psi(r) dr=$$
$$\int\limits \mF_\psi(\mW_\Psi(f))(-r)\psi(r) dr=\mW_\Psi(f)(1),$$ 
where $\mW_\Psi(f)$ is considered as a function on $\mS_c(F^\times )$
via $b\mapsto \mW_\Psi(f)(t(b))$ and 
$\mF_\psi:\mS_c(F)\rightarrow \mS_c(F)$ denotes the one-dimensional
Fourier transform with respect to $\psi$ and the self-dual measure $dx$ on $F$.
   
\end{proof}

\begin{Thm}\label{main:sl2}
There exists a unique unitary operator
$\Phi_s\in \Aut(L^2(X,\omega_X)),$ that preserves the space $\mS_0(X)$ 
and satisfies
\begin{equation}\label{request:sl2}
  \left\{
 \begin{array}{ll} 
   \theta(g,t)\circ \Phi_s=\Phi_s\circ \theta(g,t^s)& g\in G_1, t\in T_1\\
   \Phi_s\circ \Phi_s=\Id\\
   \kappa_\Psi(\Phi_s)(\varphi)=s\cdot \varphi  & \varphi\in \mS_0(T_1)  
\end{array}\right.
\end{equation}
\end{Thm}

\begin{proof}
  The injectivity of $\kappa_\Psi$ implies the uniqueness of the operator
  $\Phi_s$, hence  it is enough to construct such an operator. 
  
  We define  $\Phi_s$ to be $\Phi$. The properties follow from
  Propositions \ref{SL2:Phi:properties}, \ref{SL2:S0}, \ref{SL2:kappa}.

\end{proof}

\section{$G_1=SU_3$}\label{sec:SU3}
\subsection{The structure and compatibility of measures}
\subsubsection{The field}
Let $K$ be a quadratic field extension over $F$ with the Galois
involution $x\mapsto \bar x$, the norm $\N$ and the trace $\T$.
We write $|\cdot|_K$ for the absolute value on $K$, such that 
$|x|_K=|\N(x)|_F$. We fix an element $\tau \in \mO_F$
such that $\mO_K=\mO_F+\st \mO_F.$

The space $K$ admits a quadratic form $x\mapsto \N(x)$ and
the associated  bilinear form on $K$ is $(x,y)\mapsto \T(x\bar y).$

We fix on $K$ a self dual measure $dx$ with respect to
$\psi$ and $\N$. The Fourier transform on $K$ is denoted by
$\mF_{\psi,K}$, to distinguish it from the Fourier transform
$\mF_\psi$ with respect to $\psi$ and the self-dual measure on $F$.

\subsubsection{The unitary group}

Let $(\bW,h)$ be  the following Hermitian space 
$$\bW=K^3, \quad   h(v_1,v_2)=x_1\bar z_2+ y_1\bar y_2+z_1\bar x_2,
\quad v_i=(x_i,y_i,z_i).$$

The group $G_1=SU(\bW,h)$ is the group of  automorphisms of $\bW,$
acting on the right, preserving the Hermitian form $h$
and having determinant $1$. Its elements are $3\times 3$ matrices
over $K$.

We denote by $B_1=T_1\cdot U_1$ the Borel subgroup of $G_1,$
preserving the line $K(0,0,1)$ in $\bW$. The unipotent radical
$U_1$ is the stabilizer of the vector $(0,0,1)$.
The space $X=U_1\backslash G_1$ is naturally identified with the set
$\bW^0$ of $h$-isotropic non-zero vectors in the space $\bW$.
We write $T'$ for the maximal split torus of $T_1$.

\subsubsection{The measures}
The space $\bW$ with $\dim_F(\bW)=6$
admits the $F$-bilinear form $\la v_1,v_2\ra= \T h(v_1,v_2)$
and the corresponding   quadratic form $q$ is given by
$$q(v)=\la v,v\ra /2=\T(x\bar z)+\N(y),\quad v=(x,y,z).$$

We fix the self-dual measure $dw$ on $\bW$  with respect to
$\psi$ and $q$. It gives rise to a measure on the cone
$\bW^0$ and hence to a measure on $X$ which we denote by $\omega_X$.



We fix  bijections
$$x:K\times \sqrt\tau F\rightarrow U_1^{op}, \quad t: K^\times \rightarrow T_1$$  by
$$x(r,y)=\left(\begin{array}{ccc} 1& 0 & 0\\
-\bar r& 1 &0 \\
-\frac{\N(r)}{2}+y&  r & 1 \end{array}\right)
\quad t(a)=\left(\begin{array}{ccc} a& 0&0 \\
0& a^{-1}\bar a& 0\\ 
0&  0 & \bar a^{-1} \end{array}\right)
\quad a\in K^\times.  $$
We also fix a  representative $n_s$ of the Weyl element $s$ by
$$ 
n_s=\left(\begin{array}{ccc} 0&0& 1 \\  0& -1&0\\1 & 0 &0 \end{array}\right).
$$

The  Haar measures on $K\times \sqrt \tau F$ and $K^\times$
define the measures on $U_1^{op}$ and $T_1$ respectively. 

By Bruhat decomposition for $G_1,$ there is an embedding
$j: T_1\times U_1^{op}\rightarrow X$ with dense image, defined by
$j(t,u)=[t^{-1}u]$.

It is straightforward to check that for any $f\in \mS_c(X)$ one has
$$\int\limits_X f(v) \omega_X(v)=
\int\limits_{K}\int\limits_{\sqrt{\tau}F} \int\limits_{K^\times}
f([t(b)^{-1} x(r,y)]) |\N(b)|^2 d^\times b\,dy\, dr.$$

The root system with respect to the torus $T'$ is 
$$R(G_1,T')=\{\pm \alpha, \pm 2\alpha\},\quad\alpha(t(a))=a,
\quad \forall a\in F^\times.$$

The operator $\Phi_s$ for the group $G_1$ is defined using
the normalized Radon transform on the cone $X$. Below we recall the definition
and the relevant properties. We refer to \cite{GurevichKazhdan} for  proofs.

\subsection{Mellin transform}
Let $\chi$ be a character of $\mO^\times$,  extended
to $F^\times$ by setting $\chi(\varpi)=1$.
We write $\chi_s$ for the character $\chi|\cdot|^s$
of $F^\times$. The character $\chi_s$ is lifted to the character
of $T'\simeq F^\times$ via isomorphism $t(x)\mapsto x$. 

Define the Mellin transform 
$P(\chi,s):\mS_c(X)\rightarrow \mS^\infty(X)$  along $T'$ by
$$P(\chi,s)=\int\limits_{T'} \theta(t)f \chi_s(t) dt.$$ 
The image $\mS(\chi,s)$ consists of functions $f\in \mS^\infty(X)$
satisfying ${\theta(t)f=\chi_s^{-1}(t)f}$. 

The Mellin transform can be also computed on functions on $\mS^\infty(X),$
not necessarily of compact support, provided the integral converges.

The following statement is obvious and will be used later. 

\begin{Lem}\label{Mellin:factor}
 Let $G:\mS_c(X)\rightarrow \mS(\chi,s),$
  such that $G\circ\theta(t)=\theta(t^{-1})\circ G$ for all $t\in T'$.
Then $\Ker G$ contains $\Ker P(\chi^{-1},-s).$  
\end{Lem}

\subsection{The Radon transform}
Recall that $X$ can be identified with the space $\bW^0$ of non-zero isotropic
vectors in $\bW$.  In this section elements in $X$ will be denoted
by $u,v,w\ldots$, isotropic vectors
in $\bW$.

For any  vector
$w\in \bW^0=X$ consider an algebraic map
$$p_w:\bW^0\rightarrow F, \quad p_w(v)=\la v,w\ra.$$
The measure $\omega_X$ defined above 
and the measure  $dx$ on $F$ give rise to well-defined measure
$\omega_{w,a}$ on the fiber ${p^{-1}_w(a)=\{v\in \bW^0, \la v,w\ra=a\}}$
for any $a\in F$.

For any $a\in F$ we define  Radon transform
$\mR(a):\mS_c(X)\rightarrow \mS^\infty(X)$ by
$$\mR(a)(f)(w)=
\int\limits_{p^{-1}_{w}(a)} f(v) \omega_{w,a}(v).$$
The function $a\mapsto \mR(a)(f)(w)$ is continuous, of bounded support. 
The normalized Radon transform on $\mS_c(X)$ is defined by 
$$\hat\mR(f)(w)=\int\limits_F \mR(a)(f)(w)\psi(a) da.$$
In addition set
$$\mR_1(f)(w)=\int\limits_{F^\times} \theta(t(x))f(w)\chi_K(x)d^\times x$$

Below we list the properties of the operators $\mR(a)$
and $\hat \mR$, all proven in \cite{GurevichKazhdan}, section $3$.
The quadratic space $(V_K,q_K)$ in loc. cit. is isomorphic to the
quadratic space $(\bW,q)$ and  the results proven in loc.cit. hold
in our setting. 

For all $f\in \mS_c(X)$, $w\in X$ one has
\begin{enumerate}
  
\item $\mR(xa)(f)(xw)=|x|^{-1}\mR(a)(f)(w)$ for all $x\in F^\times$.
  This implies
  $$\hat\mR(f)(xw)=\int\limits_F \mR(a)(f)(w)\psi(ax) da.$$

\item $\mR(a)\circ \theta(g,t)=\theta(g,t^{-1})\circ \mR(a)$
for $g\in G_1,t\in T'$  and the same is true for $\hat \mR$.

\item There exists a constant $c_{\psi,q}$ such that
  for $|a|$ small enough one has 
  $$\mR(a)(f)(w)=\mR(0)(f)(w)+c_{\psi,q}\chi_K(a)|a|\mR_1(f)(w).$$

\item The function $x\mapsto \theta(t(x))\hat\mR(f)(w)$
  is bounded  for $x\in F$.

\item $\hat \mR(f)$ extends to a locally constant function on
  $\bW^0\cup \{0\}$  whose value at $0$ is $\int\limits_X f(v)\omega_X(v)$.
  
\end{enumerate}

 \begin{Lem}\label{kernel:lemma}
   \begin{enumerate}
     \item
       If $f\in \Ker P(\chi_K,z)$ then
       $\hat\mR(f)\in \Ker P(\chi_K,-z)$
       for $Re(z)>0$.
\item
Let $f\in \Ker P(\chi_K,-1)\cap \Ker P(\chi_K,0).$
For  $w\in X$ the function $a\mapsto \mR(a)(f)(w)$
is of compact support on $F^\times$. 
   \end{enumerate}
 \end{Lem}
 
 \begin{proof}
   \begin{enumerate}
     \item
The transform $P(\chi_K,-z)(\hat \mR(f))(w)$ is well-defined for $Re(z)>0$
by the property $(4)$. The Lemma \ref{Mellin:factor} yields the result.
\item
   By properties $(1)$, $(2)$,  the map $\mR(0)$ has image in $\mS(\chi_0,1)$
   and satisfies the condition of Lemma \ref{Mellin:factor}.
   Similarly, the map $\mR_1$ has image in $\mS(\chi_K,0)$
   and satisfies the condition. Hence for
   $f\in   \Ker P(\chi_K,-1)\cap \Ker P(\chi_K,0)$ one has
   $\mR(0)(f)=\mR_1(f)=0$. By the property $(3),$
the function $\mR(a)(f)(w)$ vanishes
   for small $|a|$ and hence is of compact support on $F^\times$. 
\end{enumerate}
   \end{proof}

 Let us fix terminology for convergence of integrals
 of locally constant functions, not necessary of compact support, on $F^\times$.
 For $f\in \mS^\infty(F^\times)$ we say that $\int\limits_{|x|\le 1} f(x) d^\times x$
 
 \begin{itemize}
 \item converges absolutely if $\int\limits_{|x|\le 1} |f(x)| d^\times x$ converges
 \item converges if  $\lim_{n\to \infty} \int_{|x|\ge q^{-n}} f(x) d^\times x $ exists. 
\item stabilizes if the sequence $\int_{|x|\ge q^{-n}} f(x) d^\times x$ becomes constant for $n>N$.
 \end{itemize}

 Similarly we say that the integral $\int\limits_{|x|> 1} f(x) d^\times x$ converges  absolutely,
 (resp. converges or stabilizes) the integral  if  $\int\limits_{|x|<1} f(x^{-1}) d^\times x$
 converges  absolutely,(resp. converges or stabilizes).

 Given an integral $I=\int_{F^\times} f(x) d^\times x$ we say
 that it stabilizes at zero and converges absolutely
 at infinity if  $\int\limits_{|x|\le  1} f(x) d^\times x$ stabilizes
 and  $\int\limits_{|x|>1} f(x) d^\times x$ converges absolutely. 
 
 For example, for any unitary character $\chi$
 and $Re(s)>0$ the integral $\int_{F^\times}\psi(x)\chi(x)|x|^s d^\times x$ stabilizes
 at infinity and converges absolutely at zero. 
 
\subsection{The operators $\Phi$}
We define an operator 
${\Phi:\mS_c(X)\rightarrow \mS^\infty(X)}$  by
  $$\Phi(f)=\int\limits_{F^\times}
\theta(t(x))\hat\mR(f)\psi(x^{-1}) \chi_K(x)|x|^{-1} d^\times x.$$

By properties $(4),(5)$ the integral converges absolutely.
By property $(2)$ it satisfies the equivariance property for
$G\times T'$.

In \cite{GurevichKazhdan}, using the minimal representation
for the group $O(8)$ we proved

\begin{Thm}\label{Phi:involution}
  The operator $\Phi$
  \begin{enumerate}
    \item has its image in the space of functions of bounded support on $X$.
  \item extends to a unitary involution on $L^2(X,\omega_X),$
  \item satisfies $\theta(g,t)\circ\Phi=\Phi \circ\theta(g,t^s)$ for
  all $g\in G_1, t\in T'$. 
    \end{enumerate}
  \end{Thm}

The operator $\Phi$ is our candidate for Fourier transform.
To prove Theorem \ref{main} for $G_1$ it remains
\begin{itemize}
\item to show that $\Phi$ enjoys the equivariance
  property with respect to $T_1$,
\item to define a space $\mS_0(X)\subset \mS_c(X),$
  preserved by $\Phi$ and dense in $L^2(X,\omega_X)$ and
  \item to compute $\kappa_\Psi(\Phi)$ on the space
$\mS_0(T_1)=\mW_\Psi(\mS_0(X))$. 
\end{itemize}

  \subsection{ The space $\mS_0(X)$}\label{def:s0:su3}
  Define the space $\mS_0(X)=\mS_\bB(X),$
  where $\bB$ is the following
  finite set of characters of $T'\simeq F^\times$:
 $$\bB=\{\chi_K, \chi_K|\cdot|^{\pm 1},|\cdot|^{\pm 2}\}.$$

\begin{Prop}\label{su3:s0:preserved}
    The operator $\Phi$ preserves $\mS_0(X).$ 
\end{Prop}

\begin{proof} We start by showing that for $f\in \mS_0(X)$
  one has $\Phi(f)\in \mS_c(X)$.
  Since $\Phi(f)$ has bounded support, it is enough to show that
  the germ $[\Phi(f)]_0$ at zero vanishes.
  
The operator $\Phi$ can be naturally decomposed as a sum
$\Phi=\Phi_1+\Phi_2$ where
$$\Phi_1(f)(w)=\gamma(\chi_K,\psi)
\int\limits_{F^\times} \theta(t(x))\hat \mR(f)(w)\chi_K(-x)|x|^{-1} d^\times x, $$
and
$$\Phi_2(f)(w)=\gamma(\chi_K,\psi)
\int\limits_{F^\times} \theta(t(x))\hat \mR(f)(w)
(\psi(x^{-1})-1)\chi_K(-x)|x|^{-1} d^\times x. $$

For $f\in \mS_0(X)$ one has $\Phi_1(f)=0$ by \ref{kernel:lemma}.
Let us show that the germ $[\Phi_2(f)]_0$ is zero for $f\in \mS_0(X)$.
The function
$$g(x)=\gamma(\chi_K,\psi)(\psi(x^{-1})-1)\chi_K(-x)|x|$$ has a bounded support,
denote it by $\mB$.
For $|w|$ small enough, the function $x\mapsto \hat \mR(f)(xw)$ is constant
for $x\in \mB$. Hence for $|w|$ small one has
$$\Phi_2(f)(w)=\hat \mR(f)(w)\cdot \int_{F^\times} g(x) dx.$$

By property $(5)$  $\hat \mR(f)(w)=\int\limits_X f(v)\omega_X(v)$
for $|w|$ small enough. 
The map $f\mapsto\int\limits_X f(v)\omega_X(v)$ has its image in
$\mS(\chi_0,|\cdot|^{-2})$. Hence by Lemma \ref{Mellin:factor} 
if $f\in \Ker P(\chi_0,|\cdot|^2)$ then $\hat \mR(f)(w)=0$ for small $w$ and so
$[\Phi_2(f)]_0=0$.  Hence $\Phi(f)$ is of compact support. 

Since $\Phi\circ \theta(t)=\theta(t^{-1})\circ \Phi$ for $t\in T'$
the properties $f\in \Ker P(\chi,s),$ and $\Phi(f)\in \mS_c(X)$
imply $\Phi(f)\in \Ker P(\chi^{-1},-s)$, as in Proposition
\ref{SL2:S0}.
This yields the result. 
\end{proof}
  
  \begin{Prop}\label{su3:kernel}
    For $f\in \mS_0(X)$ one has
    $$\Phi(f)(w)=\int\limits_X f(v)\mL(\la v, w\ra) \omega_X(v),$$
    where for $a\in F^\times$
\begin{equation}\label{def:L}
    \mL(a)=
  \gamma(\chi_K,\psi)
  \int\limits_{F^\times} \psi(ax+x^{-1})\chi_K(-x)|x| d^\times x.
  \end{equation}
\end{Prop}

  \begin{proof}
    For $a\in F^\times$, the integral defining $\mL$ stabilizes both at zero at and infinity. In particular, there exists a compact set $\mK_1$ in $F^\times$
 such that
$$\mL(a)=
  \gamma(\chi_K,\psi)
  \int\limits_{\mK_1} \psi(ax+x^{-1})\chi_K(-x)|x| d^\times x.$$
    
 For $f\in \mS_0(X)$, the function $a\mapsto \mR(a)(f)(w)$
 is of compact support on $F^\times$ by Lemma \ref{kernel:lemma}, part $(2)$. 
 We can assume that the support is contained in $\mK_1$.

By the Fubini theorem 
$$\int\limits_X f(v)\mL(\la v, w\ra) \omega_X(v)=
\int\limits_F \mR(a)(f)(w)\mL(a) da=$$
$$\gamma(\chi_K,\psi)\int\limits_{\mK_1} \int\limits_{\mK_1}
\mR(a)(f)(w)\psi(ax)\psi(x^{-1})\chi_K(-x)|x|d^\times x da.$$
We can change the order of integration over compact sets. This gives
$$\gamma(\chi_K,\psi)\int\limits_{\mK_1}
\left(\int\limits_{\mK_1}
\mR(a)(f)(w)\psi(ax)da\right)\psi(x^{-1})\chi_K(-x)|x|d^\times x da=$$
$$\gamma(\chi_K,\psi)\int\limits_{F^\times }\theta(t(x))\hat \mR(f)(w)
\psi(x^{-1})\chi_K(-x)|x|^{-1}d^\times x da=\Phi(f)(w),$$
as required.
\end{proof}

  \begin{Prop}\label{su3:equivariance}
    One has $\Phi\circ \theta(t^s)(f)=\theta(t)\circ \Phi(f)$
    for all $f\in \mS_0(X)$ and $t\in T_1$.
  \end{Prop}

\begin{proof}
  This is a straightforward computation and is very similar to
  the proof of the equivariance property of the classical Fourier transform. 
  $$\theta(t)\Phi(f)(w)=
 \delta^{1/2}_B(t) \int\limits_X  f(v) \mL(\la v, tw\ra) \omega_X(v).$$
One has  $\la v,tw\ra=\la (t^s)^{-1}v, w\ra$ for all $t\in T_1.$
Applying the change of variables $v\mapsto (t^s)^{-1}v$ and taking
the measure into account, we get
that the integral equals 

 $$ \delta^{1/2}_B(t^s)\int\limits_X
 f(t_s v) \mL(\la v, w\ra) \omega_X(v)=
 \Phi(\theta(t^s)f)(w)$$
 as required. 
\end{proof}

\subsection{The Whittaker map}\label{subsection:su3:whittaker}
It remains to compute $\kappa_\Psi(\Phi)$. 

We fix a character $\Psi$ of $U^{op}_1$
such that ${\Psi(x(r,r'))=\psi(\T(r))}$. 
The Whittaker map $\mW_\Psi:\mS_c(X)\rightarrow \mS_c(T_1)$
is defined as in introduction. 

\begin{Prop}\label{SU3:kappa} Let $f\in \mS_0(X)$. 
\begin{enumerate}
\item $\mW_\Psi(\Phi(f))(1)=\mW_\Psi(f)(t(-1)),$
\item
  $\mW_\Psi(\Phi(f))(t)= \mW_\Psi(f)(t(-1)t^s).$
\end{enumerate}
  \end{Prop}

The proof occupies the rest of this subsection. We start with
the following technical Lemmas, whose proofs are postponed
to the end of this subsection.  

\begin{Lem}\label{int:y}
 For any $x\in F$  and  $g\in \mS_c(K)$ one has   
$$\int\limits_{\sqrt{\tau}F} \int\limits_{K}  g(b)
  \psi(-x\T(b\cdot y)) db\, dy=
  |x|^{-1} \int\limits_F g(b) db.$$
\end{Lem}

According to Weil, \cite{Weil}
there exists a constant $\gamma(\chi_K, \psi),$ which
is  a fourth root of unity, satisfying
\begin{equation}\label{weil:K}
\int\limits_K \mF_{\psi,K}(f)(x) \psi(\N(x))dx=
\gamma(\chi_K,\psi)\int\limits_K f(x)\psi(-\N(x)) dx.
\end{equation}
For all $t\in F^\times$ denote by $\psi_t$ the additive character
$\psi_t(x)=\psi(tx)$. One has 
\begin{itemize}
  \item
$\gamma(\chi_K,\psi_t)=\chi_K(t)\gamma(\chi_K,\psi),$
\item
$\gamma(\chi_K, \psi)=1$ if $K$ is split. 
\end{itemize}

\begin{Lem} \label{int:r}
  For any $g\in \mS_c(F)$ one has 
\begin{equation}\label{eq:int:r}
  \int\limits_K \int\limits_{F^\times}
  g(x)\psi\left(-\frac{\N(r)}{x}\right)\chi_K(x)d^{\times}x
  \psi(\T(r)) dr=\gamma(\chi_K,\psi)\chi_K(-1)
  \mF_\psi(g)(1).
  \end{equation}
\end{Lem}

\begin{proof}[Proof of Proposition \ref{SU3:kappa}] 
  It is easy to see that the first part implies the second.
  Indeed, assuming part $(1)$, for any $t\in T_1$ and $f\in\mS_0(X)$ one has
  $$\mW_\Psi(\Phi)(f)(t)=\mW_\Psi(\theta(t)\Phi(f))(1)=
 \mW_\Psi(\Phi(\theta(t^s)f))(1)=$$
  $$\mW_\Psi(\theta(t^s)f)(t(-1))=\mW_\Psi(f)(t(-1)t^s).$$

Using Bruhat decomposition for $G_1$ this equals.
$$\mW_\Psi(\Phi(f))(1)=
 \int\limits_{U_1^{op}}
  \int\limits_{T_1}
  \int\limits_{U_1^{op}}
  f([t^{-1}u_1])\mL(\la[t^{-1}u_1],[u_2]\ra)\Psi(u_2)^{-1}
  \delta_B(t)\,du_1\, dt \,du_2=$$
  $$\int\limits_{U_1^{op}}
  \int\limits_{T_1}
  \int\limits_{U_1^{op}}
  \theta(t)f([u_2])\Psi(u_2)^{-1} du_2
  \mL(\la[t^{-1}u_1],[1]\ra)\Psi(u_1)
  \delta^{1/2}_B(t) dt\,du_1=$$
  $$\int\limits_{U_1^{op}}
  \int\limits_{T_1} \mW_\Psi(f)(t)
  \mL(\la[t^{-1}u_1],[1]\ra)\Psi(u_1)
  \delta^{1/2}_B(t)dt\,du_1.$$
  We put
  $$t=t(b), b\in K^\times,
  \quad u_1=x(r, y), r\in K, y\in \sqrt\tau F.$$
To ease  notation we write $\bar f\in\,S_c(F^\times)$ for
  the function $b\mapsto \mW_\Psi(f)(t(b)).$

  One has 
$$ \la  [t^{-1}(b)x(r,y)],[1] \ra=
  -\T(b) \N(r)/2- \T(b y).$$
Hence the above equals  
  $$\int\limits_{K}\int\limits_{\sqrt\tau F}
  \int\limits_{K^\times} \bar f(b)
  \mL(-\T(b)\frac{\N(r)}{2}-\T(b\sqrt{\tau}y))\psi(\T(r))
  |\N(b)| d^\times b dr ty.$$
 
  Writing explicitly the expression for $\mL$ from \ref{def:L}
  and rearranging the change of integrals this equals
%
\begin{multline}
\gamma(\chi_K,\psi)\chi_K(-1)\cdot 
\int\limits_{K}\int\limits_{F^\times} \\
\int\limits_{\sqrt \tau F} \int\limits_{K^\times} 
\left(\bar f(b) \psi(-\T(b x \N(r)/2)) \right)
\psi(-\T(bx y )) d b dy\\
\chi_K(x)\psi(x^{-1})|x| d^\times  x  \,\, \psi(\T(r)) dr.
\end{multline}

We apply  Lemma \ref{int:y} to the middle line, i.e. for 
$$g(b)=\bar f(b)\psi(-\T(b x \N(r)/2)).$$
Notice that for $b\in F$ one has  $\T(b x\N(r)/2)=bx\N(r).$
Hence the middle line equals
$$|x|^{-1}\int\limits_F \bar f(b)\psi(-bx\N(r)) db $$
The integral becomes  $\gamma(\chi_K,\psi)\chi_K(-1)$ times 

$$
  \int\limits_{r\in K} \int\limits_{b\in F} \bar f(b)
  \int\limits_{x\in F^\times} \psi(-bx \N(r)) \chi_K(x)\psi(x^{-1}) d^\times x db \,\, \psi(\T(r)) dr
  $$

  After the change of variables $bx\mapsto x^{-1}$ this becomes
 $\gamma(\chi_K,\psi)\chi_K(-1)$ times 
$$ \int\limits_{r\in K}
\int\limits_{x\in F^\times}
\left(\int\limits_{b\in F} \bar f(b)\chi_K(b) \psi(xb) db\right)
\psi(-\frac{\N(r)}{x})
\chi_K(x) d^\times x  \,\, \psi(\T(r)) dr=
$$

$$\gamma(\chi_K,\psi)\chi_K(-1)\int\limits_{r\in K}
\int\limits_{x\in F^\times} \mF_\psi(\bar f\chi_K)(x)
\psi(-\frac{\N(r)}{x})
\chi_K(x) d^\times x  \,\, \psi(\T(r)) dr.
$$

Applying Lemma
\ref{int:r} to $g=\bar f \chi_K$  and the properties of
$\gamma(\chi_K,\psi)$ we obtain that $\mW_\Psi(\Phi(f))(1)$ equals
$$(\gamma(\chi_K,\psi)\chi_K(-1))^2
\mF_\psi(\mF_\psi(\bar f\chi_K))(1)=\bar f(-1)=\mW_\Psi(f)(t(-1)),$$
as required.
\end{proof}

It remains to prove Lemmas.
\begin{proof}[Proof of Lemma \ref{int:y}]
 We fix the isomorphism of vector spaces
  $$K\simeq F\oplus F, \quad b_1+\sqrt{\tau}b_2\mapsto (b_1,b_2)$$
which induces the isomorphism $S_c(K)\simeq S_c(F)\otimes S_c(F)$.
The self-dual measure on $K$ with respect to $(\psi, \N)$
is transported under this isomorphism to $|2| |\tau|^{1/2} db_1 db_2.$

It is enough to prove Lemma for $g=g_1\otimes g_2,$
where $g_1, g_2\in \mS_c(F),$ so that $g(b_1+\sqrt{\tau}b_2)=g_1(b_1)g_2(b_2).$

Let us write $y=\sqrt{\tau}y'$ for $y'\in F$ and $dy=|\tau|^{1/2} dy'.$
Then for
$b=b_1+\sqrt{\tau} b_2$ one has $\T(bx\sqrt{\tau}y')=2\tau b_2 xy'.$

$$\int\limits_F \int\limits_K g(b) \psi(\T(bx\sqrt{\tau}y) db dy=$$

 $$\int\limits_{F^3} g(b_1,b_2)\psi(2\tau b_2xy)  |2\tau| db_1 db_2 dy'=
 |2\tau|
 \int\limits_F g_1(b_1)db_1 \int\limits_F \mF_{\psi}(g_2)(2\tau xy)dy=$$
 $$g_2(0)|x|^{-1}\int\limits_F g_1(b_1)db_1=|x|^{-1}\int\limits_F g(b)db.$$

\end{proof}

\begin{proof}[Proof of Lemma \ref{int:r}]
Let $c\in F^\times\backslash \N(K^\times)$, so that
$F^\times=\N(K^\times)\cup c\N(K^\times)$.
The measures $d^\times y$ on $K^\times$
and $d^\times x$ on $\N(K^\times)\subset F^\times$ define 
Haar measures  on the fibers of $\N:K^\times \mapsto F^\times$.
All the fibers are compact and have the same measure $C$.
By the Fubini theorem for any function $h\in L^1(F^\times)$ one has  
\begin{equation}\label{int:F:K}
 \int\limits_{F^\times} h(x) d^\times x=
C^{-1} \int\limits_{K^\times} h(\N(y))+h(c\N(y)) d^\times y.
\end{equation}

Applying this
  integral over $F^\times$ in the LHS of \ref{eq:int:r}
  we obtain 
  $$C^{-1}\int\limits_K \int\limits_{K^\times}
  g(\N(y))\psi(-\N(r/y))- g(c\N(y))\psi(-c^{-1}\N(r/y))d^{\times}x
  \psi(\T(r)) dr. $$
After the change of variables $r\mapsto r\bar y$ this equals
  \begin{multline*}
    C^{-1}\times 
    \int\limits_K
    \left(
    \int\limits_{K} g(\N(y))\psi(\T (r \bar y))\,dy\right) \psi(-\N(r))dr -\\
    \int\limits_K
    \left(\int\limits_{K} g(c\N(y))\psi(\T (r \bar y))\, dy\right)
    \psi(-c^{-1}N(r)) dr=
  \end{multline*}
  \begin{multline*}
    C^{-1}\times 
    \int\limits_K \mF_{K,\psi}(g\circ \N)(r)  \psi(-\N(r))dr-
    \int\limits_K \mF_{K,\psi}(g_c\circ \N)(r) 
    \int\limits_{K} \psi(-c^{-1}\N(r)) dr
  \end{multline*}
where $g_c(x)=g(cx)$ for any $x$. This equals by \ref{weil:K}
  \begin{multline*}
    C^{-1}\times \chi_K(-1)\gamma(\chi_K,\psi)
    \int\limits_K (g\circ \N)  \psi(\N(r))dr+
    |c|\int\limits_K (g_c\circ \N)(r) \psi(c\N(r)) dr.
  \end{multline*}
  Applying  equation \ref{int:F:K} again this equals
  $$\chi_K(-1)\gamma(\chi_K,\psi)\int\limits g(x)\psi(x) dx= 
  \gamma(\chi_K,\psi)\chi_K(-1)\mF_\psi(g)(1).$$

\end{proof}

The restriction $\mW_\Psi:\mS_0(X)\rightarrow \mS_c(T)$,
whose image we denote by $\mS_0(T),$
gives rise to the homomorphism
$\kappa_\Psi:\End_{G}(\mS_0(X))\rightarrow \End_{\C}(\mS_0(T))$.
\begin{Lem} $\kappa_\Psi$ is injective.
\end{Lem}
\begin{proof} See the proof of \ref{injective:general}
  for the general case.
  \end{proof}

Let us define the action of $W$ on $\mS_0(T_1)$.
\begin{Def}\label{SU3:W:action}
The action of $W$ on $\mS_0(T_1)$ is defined by 
$$s\cdot \varphi(t)= \varphi(t(-1)t^s)$$
\end{Def}

\begin{Thm}\label{main:su3}
There exists a unique  unitary  involution 
$\Phi_s\in \Aut(L^2(X,\omega_X))$  that preserves the space $\mS_0(X)$ 
and satisfies
\begin{equation}\label{request:su3}
  \left\{
 \begin{array}{ll} 
   \theta(g,t)\circ \Phi_s=\Phi_s\circ \theta(g,t^s)& g\in G_1, t\in T_1\\
   \kappa_\Psi(\Phi_s)(\varphi)=s\cdot \varphi  & \varphi\in \mS_0(T_1)  
\end{array}\right.
\end{equation}
\end{Thm}

\begin{proof}
  The injectivity of $\kappa_\Psi$ implies uniqueness of such operator.
  and hence it is enough to construct such $\Phi_s$.
  We put $\Phi_s=\Phi$. The properties  follow from
  Theorem \ref{Phi:involution},
  Propositions \ref{su3:equivariance},
  \ref{su3:s0:preserved} and \ref{SU3:kappa}, part $(2)$.
\end{proof}

\section{Quasi-split groups}\label{sec:qs:groups}

 We recall  below the structure of
 reductive quasi-split groups. Our main reference is \cite{BruhatTits}.

\subsection{Relative and absolute root systems} 
Let ${\bf G}$ be a reductive,  connected, simply-connected
quasi-split group over $F$ with a maximal split torus ${\bf T'}$.
We denote by $Lie(G)$ the Lie algebra of ${\bf G}$ an by $Ad$ the adjoint
action of ${\bf G}$ on $Lie(G)$.
Let  ${\bf T}$ be the centralizer  of ${\bf T'}$ and ${\bf N}$ be the normalizer of ${\bf T'}$,
both defined over $F$.

The root datum of ${\bf G}$ with respect to ${\bf T'}$
is a quadruple $(X^\ast({\bf T'}), R, X_\ast({\bf T'}), R^\vee)),$
where the set of roots $R\subset X^\ast({\bf T'})$ consists of the
weights that appear in the representation
$Ad:{\bf T'}\rightarrow Aut(Lie(G))$.

The root system $R$ is not necessarily reduced. 
For any root $\alpha\in R$, its root ray is defined as
$1\otimes R\cap \bR_{>0}\otimes \alpha,$
in $\bR\otimes X^\ast({\bf T'})$. Each root ray contains one or two elements.
 We denote by ${\bf R}$ the set of root rays.

The choice of a Borel subgroup ${\bf B}$,
containing ${\bf T}$ and defined over $F$ determines
the decomposition $R=R^+\cup R^-$ into the set of positive and negative roots
and the subset $\Delta\subset R^+$ of simple roots.
We call a root ray positive (resp. negative, resp. simple) if it contains
a positive (resp. negative, resp. simple) root.

The groups  ${\bf G}$ and ${\bf T}$ are split over  the separable closure
$F_s$ of $F$.  There exists a minimal extension $F\subset E\subset F_s$
over which ${\bf T}$ and hence ${\bf G}$ splits. Then $E/F$ is Galois.
We denote this split $E$-group by $\tilde{\bf G}$. It  has a root datum
$(X^\ast({\bf T}), \ti R, X_\ast({\bf T}), \ti R^\vee).$
Note that all root rays in $X^\ast({\bf T})\otimes \bR_{>0}$ are singletons. 

The Borel subgroup  $\ti {\bf B}$ containing ${\bf B}$
of $\ti {\bf G},$ determines the set
$\ti R^+$ of positive roots and the set $\ti \Delta$ of  simple roots.
The Galois group  $\Gamma=Gal(E/F)$ acts on
$X^\ast({\bf T}), \ti R, \ti R^+$ and  $\ti \Delta$.

There is a bijection $\beta\leftrightarrow \ti R_\beta$ between
the set $R$ of roots  and the set of $\Gamma$ orbits of $\ti{R}$.
The restriction of every root in $\ti R_\beta$ to $T'$ equals to $\beta$.

\begin{Def} Let  $\alpha\in \ti R$. The field
  $L_\alpha=E^{\Gamma_\alpha}$ is called the field of
  definition of $\alpha$, where  $\Gamma_\alpha\in \Gamma$ is the stabilizer
  of $\alpha$.
\end{Def}

  \begin{Prop}
    \begin{enumerate}
\item For any $\gamma\in \Gamma$ and $\alpha\in \ti{R}$
  one  has $L_{\gamma(\alpha)}=\gamma(L_\alpha).$
\item  For  $\alpha\in \ti R,$ if $\alpha|_{T'}$ is  a divisible
        root in $R$,  then there exist roots $\alpha_1,\alpha_2\in \ti R$
        such that
        $$\alpha_1|_{T'}=\alpha_2|_{T'}=\alpha/2|_{T'},
        \quad \alpha=\alpha_1+\alpha_2.$$  In addition
        $L_{\alpha_1}=L_{\alpha_2}$ is a quadratic extension of $L_{\alpha}$. 
    \end{enumerate}
    \end{Prop}

\subsection{The Chevalley-Steinberg pinning}
For any $a\in {\bf R}$ there exists a maximal
connected subgroup ${\bf U}_a$ of ${\bf G}$,  defined over $F$, such that
the weights that appear in the representation
$Ad:{\bf T'}\rightarrow \Aut(Lie({\bf U}_a))$ belong to $a$. 
The group ${\bf U}_a$ is called the root subgroup corresponding to 
$a\in {\bf R}$.


For any   simple root ray $a$ in ${\bf R},$
let ${\bf G}_a$ be the group generated by ${\bf U}_a$ and ${\bf U}_{-a}$. Since
the group ${\bf G}$ is simply-connected,
the group ${\bf G}_a$ is a simply connected  group of rank $1$ over $F$.
We denote by ${\bf T}_a$ and ${\bf T}_a'$ the maximal torus and the maximal
split torus of ${\bf G_a}$ respectively.
The group
$\ti{\bf G}_a$ in $\ti{\bf  G}$ is ${\bf G}_a$
considered as a group over $E$. 

The following proposition describes $G_a$ and $\ti G_a$.

\begin{Prop} Let $a$ be a root ray. There are two possible cases
\begin{itemize}
\item $a=\{\alpha\}$. In this case the group
  $\ti{\bf  G}_a$ is isomorphic over $E$ to a
  product of copies of the group $SL_2$, indexed by $\ti R_\alpha$.

  There exists  an isomorphism
  $\phi_a:SL_2(L_\alpha) \rightarrow G_a$
  such that
  $$\phi_a(x(r))\in U_{-a},\quad \phi_a(^tx(r))\in U_{a}, \quad\phi_a(n_s)\in N$$
 
\item $a=\{\alpha, 2\alpha\}$. In this case 
  the group $\ti{\bf  G}_a$ is isomorphic to a product of copies
  of $SL_3$ indexed by
  the set $I$ of subsets $\{\alpha_1,\alpha_2\}\subset \ti R_\alpha,$
  such that
  $\alpha_1+\alpha_2\in \ti R$. The field  $L_{\alpha_1}=L_{\alpha_2}$
  is a quadratic extension of $L{\alpha_1+\alpha_2}$
  with a non-trivial automorphism $x\mapsto \bar x$.  
  Let  $SU_3$ be the group of automorphisms
  on the Hermitian space $L_{\alpha_1}^3$ preserving the form
  $h(x,y,z)=\T(\bar x z)+\N(\bar y y)$ and having
  determinant $1$. It is a quasi-split group of rank $1$
  over $L_{\alpha_1+\alpha_2}$.

  There exists an isomorphism
  $\phi_a:SU_3(L_{\alpha_1+\alpha_2}) \rightarrow G_a$
  such that
  $$\phi_a(x(r,r'))\in U_{-a},\quad\phi_a({^tx(r,r')})\in U_{a},
  \quad\phi_a(n_s)\in N.$$
    
\end{itemize}
\end{Prop}

From now on we fix a  family of isomorphisms
$\phi_a, a\in {\bf R}$ such that $\phi_a$
define a Steinberg-Chevalley pinning of the group ${\bf G}$. See
\cite{BruhatTits}, page $78$.

\subsection{The Weyl group}
The Weyl group $W$ is isomorphic to $N/T$.
For any $a\in {\bf R}$ the image of the element
$n_{s_a}=\phi_a(n_s)$ in $W$
is denoted by  $s_a.$ These elements, called simple reflections, 
generate $W$.

The roots in the same $W$ orbit have the same field of definition. 

For any $w\in \ti W$ we denote by $l(w)$ the length of a reduced
presentation of $w$ as a product of simple reflections. 

For any $w\in W$ we define  $R(w)=R^+\cap w^{-1}(R^-)$.
Then $l(w)=|R(w)|$.

We denote by $w_0$ the longest element of $W$, and by $n_0$
its representative in $N$.

\subsection{The action of $W$ on $\mS_c(T)$}
\begin{Def}\label{def:ta}
 Define for any $w\in W$ the element  and 
 $$t_w=\Pi_{a \in  {\bf R}(w)} t_a\in T,$$ where
$t_a=\phi_a(t(-1))$ for $a=\{\alpha,2\alpha\}$ and $t_a=1$ otherwise.
  \end{Def}


  \begin{Lem}\label{t, Gamma: prop}  
        $$t_{w_2}\cdot (w_2^{-1} t_{w_1}w_2)=t_{w_1w_2}$$
  \end{Lem}
  \begin{proof}
The set $R(w_1w_2)$ can be written as a disjoint union
    $$R(w_1w_2)=
    \left(R(w_2)\backslash -w_2^{-1} R(w_1)\right)\cup
    \left(w_2^{-1}R(w_1)\backslash - R(w_2)\right).$$
    Indeed,
    $$R(w_2)\backslash -w_2^{-1} R(w_1)=
    \{\alpha>0, w_2\alpha<0, w_1w_2\alpha<0\},$$
   $$w_2^{-1}R(w_1)\backslash - R(w_2)= \{\alpha>0, w_2\alpha>0, w_1w_2\alpha<0\}.$$ 
  and the union is   $R(w_1w_2)$.
Besides $R(w)=-wR(w^{-1}).$
  
Writing by definition
$$t_{w_2}=\Pi_{R(w_2)\backslash -w_2^{-1} R(w_1)} t_a\cdot \Pi_{R(w_2)\cap -w_2^{-1} R(w_1)} t_a$$
and
$$t_{w_1} =\Pi_{R(w_1)\backslash - w_2 R(w_2)} t_a\cdot \Pi_{R(w_1)\cap - w_2R(w_2)} t_a$$
we conclude that  $t_{w_2} w_2^{-1} t_{w_1} w_2=t_{w_1w_2}$. 

    \end{proof}

\begin{Prop}\label{qs:action} 
  The map $W\times \mS_c(T)\rightarrow \mS_c(T)$ defined by
\begin{equation}\label{qs:W:action}
  w\cdot \varphi(t)=\varphi(t_w\cdot w^{-1}tw)
  \end{equation}
  is an action of $W$ on $\mS_c(T)$.
\end{Prop}

  \begin{proof}
    For $w_1,w_2\in W$ one has
    $$w_1\cdot(w_2\cdot \varphi)(t)=
    (w_2\cdot \varphi)(t_{w_1} w_1^{-1}tw_1)=$$
   $$
    \varphi((t_{w_2}\cdot w_2^{-1} t_{w_1}w_2)\cdot (w_1w_2)^{-1} t(w_1w_2)),$$
    which by Lemma \ref{t, Gamma: prop} equals $w_1w_2\cdot \varphi(t)$. 
       \end{proof} 

  For groups of rank $1$ this action was defined in
  \ref{SL2:W:action} and \ref{SU3:W:action}.

  \section{Generalized Fourier transforms}\label{sec:GFT}
In this section we generalize  Theorems \ref{main:sl2}
and \ref{main:su3} that concern the quasi-split  groups of $F$-rank one
to a general quasi-split group $G$.
We keep the notation of Section \ref{sec:qs:groups}.

For any root ray  $a$ of the group $G$
we fix the isomorphisms ${\phi_a: G_1\rightarrow G_a}$, where
$G_1$ is a quasi-split group of rank $1$. 

To formulate the main result we introduce the spaces
$\mS_0(X),$ $\mS_0(T)$ and the homomorphism
$\kappa_\Psi:\End_G(\mS_0(X))\rightarrow \End_\C(\mS_0(T)).$

\subsubsection{The space $\mS_0(X)$.}
We define for each positive root ray $a$ a set of triples $\bB_a$
as in section \ref{sec:S0} as follows.

\begin{enumerate}
\item  
Assume that  $a=\{\alpha\}$ and $L_\alpha$ be the field of definition
of $\alpha$.
Then
$$\bB_a=\{(L_\alpha, a_i(x)=\phi_a(t(x)), \chi_{\pm}(x)=|x|^{\pm 1}_{L_\alpha})\}$$

\item Assume that
  $a=\{\alpha,2\alpha\}$ and $L_\alpha\supset L_{2\alpha}$ are the fields of definition
  of $\alpha$ and $2\alpha$.
  The set $\bB_a$ consists of the triples  $(L_i,a_i,\chi_i)$ where
  $$L_i=L_{2\alpha},\quad  a_i(x)=\phi_a(t(x)), \quad \chi_i\in
  \{\chi_{L_\alpha}, \chi_{L_\alpha}|\cdot|^{\pm 1},|\cdot|^{\pm 2} \}.$$
\end{enumerate}

\begin{Def}\label{def:S0:general}
  Define  $\mS_0(X)=\mS_\bB(X),$ where $\bB=\cup_{a} \bB_a$ and  the union is
  taken over all positive  root rays.
  \end{Def}
 In particular   $\mS_0(X)=\cap_a \mS_{\bB_a}(X).$
The Weyl group acts naturally on the set $\bB$, by
$w(L_\alpha,\phi_a\circ t,  \chi_i)=
(L_{w(\alpha)}=L_\alpha,\phi_{w(a)}\circ t ,\chi_i)$
where $\alpha\in a$. Note that  under this action $w(\bB_a)=\bB_{wa}.$

For groups of rank one, the definition
of the space $\mS_0(X)$ coincides with the definition given
in \ref{def:S0:SL2} and \ref{def:s0:su3}.

\subsubsection{Whittaker map  and the map $\kappa_\Psi$}\label{subsection:general:whitt}

We define a distinguished non-degenerate
character $\Psi:U^{op}\rightarrow \C$ that is compatible with the
fixed family of isomorphisms $\{\phi_a\}$ from section \ref{sec:qs:groups}. 

For a quasi-split group $G_1$ of $F$-rank $1$
with Borel subgroup $T_1\cdot U_1$, we define a
complex character $\Psi_1$ of $U^{op}_1$ by

\begin{itemize}
  \item
$\Psi_1(x(r))=\psi(\T_{L/F}(r))$ if $G_1=\Res_LSL_2$
  \item
$\Psi_1(x(r,s))=\psi(\T_{K/F}(r))$ if $G_1=\Res_L SU_3,$   
corresponding to a quadratic field extension $K/L$. 
\end{itemize}

Let  $\Psi$ be the unique character of $U^{op}$
such that for every simple root ray $a$
the restriction $\Psi$ to $U_{-a}$ equals $\Psi_1^a=
\Psi_1\circ \phi^{-1}_a$.

For this $\Psi$ the Whittaker map $\mW_\Psi:\mS_c(X)\rightarrow \mS_c(T)$,
defined as in the introduction,
$$\mW_\Psi(f)(t)=\int\limits_{U^{op}} \theta(t)f([u]) \Psi^{-1}(u) du$$
gives rise to an isomorphism $\mS_0(X)_{U^{op},\Psi}\simeq \mS_0(T),$
where  $\mS_0(T)=\mW_\Psi(\mS_0(X))$.
This isomorphism induces the  map 
$$\kappa_\Psi:\End_G(\mS_0(X))\rightarrow
\End_\C(\mS_0(X)_{U^{op},\Psi})=\End_\C(\mS_0(T))$$

\begin{Lem}\label{injective:general}  The map $\kappa_\Psi$ is injective.
\end{Lem}

  \begin{proof}
  Let us show that  $\Ker \mW_\Psi$ does not contain non-zero $G$-modules.
  Indeed, assume that  $V\subset \Ker \mW_\Psi\subset \mS_0(X)$
  is a non-zero $G$-module.  For any character $\chi$ of $T$ the space
  of coinvariants  $\mS_0(X)_{T,\chi^{-1}}$  is naturally isomorphic
  to the normalized  principal series representation $\Ind^G_B(\chi)$.
  The functor of coinvariants induces a map
  $V_{T,\chi^{-1}}\rightarrow \Ind^G_B(\chi).$
  For every character $\chi$ in a  Zarisky-open set one has: 
  \begin{itemize}
    \item
      for some $f\in V$ the Mellin transform
      $P_\chi(f)=\int\limits_T \theta(t)f \cdot \chi(t)dt\neq 0,$
    \item   the representation  $\Ind^G_B(\chi)$ is irreducible.
\end{itemize}
  We pick such $\chi$. Since $f$ does not belong to the kernel of $P_\chi$,
  so  the map $V_{T,\chi^{-1}}\rightarrow \Ind^G_B(\chi)$ is non-zero, thus surjective. 
  The functor of coinvariants with respect to $(U^{op},\Psi)$ is exact and hence
  there is a surjection
  $$(V_{T,\chi^{-1}})_{U^{op},\Psi}\rightarrow \Ind^G_B(\chi)_{U^{op},\Psi}.$$
  Since  $V\subset \Ker \mW_\Psi,$ one has  $0=V_{U^{op},\Psi}=
  (V_{U^{op},\Psi})_{T,\chi^{-1}},$
while $\Ind^G_B(\chi)_{U^{op},\Psi}\neq 0$. This is a contradiction.

    Let $\mB\in \End_{G}(\mS_0(X))$ such that $\kappa_\Psi(\mB)=0$.
    Then  $\mW_\Psi\circ \mB=0,$ and  $\im(\mB)$ is a $G$-module,
  contained in $\Ker \mW_\Psi$ and hence is zero.
  So $\mB=0$ and $\kappa_\Psi$ is injective. 
\end{proof}

  We have defined all the notation, mentioned in  Theorem \ref{main}.
  It states:
     
There exists a unique  family of unitary operators
$\Phi_w\in \Aut(L^2(X)), w\in W$ that preserves the space $\mS_0(X)$ 
and satisfies
\begin{equation}\label{request:general}
  \left\{
 \begin{array}{ll} 
   \theta(g,t)\circ \Phi_w=\Phi_w\circ \theta(g,t^w)& g\in G, t\in T\\
   \kappa_\Psi(\Phi_w)(\varphi)=w\cdot \varphi  & \varphi\in \mS_0(T)\\
    \Phi_{w_1}\circ \Phi_{w_2}=\Phi_{w_1w_2}& w_1,w_2\in W\\
\end{array}\right.
\end{equation}

We begin with the construction of the operators
$\Phi_{s}$ for simple reflections, based on the results
for the groups of rank one. 

\subsubsection{The definition of $\Phi_{s_a}$}
The space $L^2(X)$ is the unitary completion $L^2\text{-}\ind^G_U 1$
of the space $\mS_c(X)=\ind^G_U 1$. 

For a simple root ray $a$ of $G$ consider a parabolic subgroup
$P_a=M_a\cdot U^a$, with the derived group $P'_a=M'_a U^a,$
where $M'_a=G_a$  is  a semisimple group  of rank $1$.
We denote by $B_a=T_a\cdot U_a$ the Borel subgroup of $G_a$
and put $X_a=U_a\backslash G_a$.

Consider the isomorphism, implied by the transitivity of induction,
  $$\iota_a: L^2(X)\rightarrow  L^2\text{-}\ind^G_{P_a'} L^2(X_a).$$
  defined by $\iota_a(f)(g)([m])=f([mg]).$

The isometry $\Phi_{s}$ on $L^2(X_a)$, defined in sections
\ref{sec:SL2} and  \ref{sec:SU3} gives rise to an isometry
on $L^2(X)$ by functoriality of induction. We continue
to denote this isometry by $\Phi_{s_a}$. 

\begin{Def} The operator
  $\Phi_{s_a}\in \Aut_G(L^2(X))$ is defined  by
  $$\iota_a(\Phi_{s_a}(f))(g)=\Phi_s(\iota_a(f)(g)),\quad  f\in L^2(X), g\in G.$$
\end{Def}

\begin{Prop}\label{Phi:s:equi} 
  For any simple root ray $a$ the operator
  $\Phi_{s_a}\in \Aut(L^2(X))$
 is a unitary involution satisfying
  $\theta(g,t)\circ \Phi_{s_a}=\Phi_{s_a}\circ \theta(g,t^{s_a}).$
\end{Prop}
\begin{proof}
  The only non-trivial statement is the equivariance of $T$
  which is enough to prove for   $f\in\mS_0(X)$.
  
Consider an embedding with dense image
$$j: T_a\times U_a\hookrightarrow X_a, \quad (t,u)\mapsto t^{-1}n_{s_a} u.$$

For  $f\in \mS_0(X)$ the Fourier transform is given by 
$$\Phi_{s_a}(f)([g])=
\int\limits_{T_a}\int\limits_{U_a} f(t^{-1}n_{s_a}ug)
\mL(\la [t^{-1}n_{s_a}u], [1]\ra_{X_a}) \delta_B(t_1) dt_1 du.$$

Here $\mL=\psi$ for $a=\{\alpha\}$ and
is defined by \ref{def:L} for $a=\{\alpha,2\alpha\}.$

Assume that  $a=\{\alpha\}$.
  Using \ref{SL2:Phi} for $f\in \mS_c(X)$ one has
  $$\theta(t_1)\Phi_{s_a}(f)([g])=
  \delta_B^{1/2}(t_1)\Phi_{s_a}(f)([t_1^{-1}g])=$$
  $$=\delta_B^{1/2}(t_1)
  \int\limits_{T_a}
  \int\limits_{U_a} f([t^{-1} n_{s_a} u t_1^{-1} g])
  \mL(\la [1], [t^{-1}n_{s_a}]\ra)
  \delta_B(t) dt du=$$
  $$\delta_B^{1/2}(t_1)\delta^{M_a'}_{B_a}(t_1)^{-1}
  \int\limits_{T_a}
  \int\limits_{U_a} f((t_1^{s_a})^{-1} t^{-1} n_{s_a} u  g)
   \psi(\la [1], [t^{-1}n_{s_a}]\ra)
 \delta_B(t) dt du=$$
  $$\Phi_{s_a}(\theta(t_1^{s_a})f)([g]).$$

  We have used the fact that the inner product is $G_a$ invariant and that 
$$\delta_B^{1/2}(t_1)\delta^{M_a'}_{B_a}(t_1)^{-1}
=\delta_B^{1/2}(t_1^{s_a}).$$

\end{proof}

\begin{Prop} \label{Phi:s:S0} For any simple root ray $a$ the operator
  $\Phi_{s_a}$ preserves $\mS_0(X)$.
\end{Prop}

\begin{proof}
  We have defined for any root ray $a$ the set of triples $\bB_a$ such that
  $\mS_0(X)\subset \mS_{\bB_a}(X)\subset \mS_c(X)$.
  In fact $\mS_{\bB_a}(X)=\ind^G_{P'_a} \mS_0(X_a)$ which is preserved by
  $\Phi_{s_a}$ by Definition \ref{def:S0:general} and by Propositions
  \ref{main:sl2}, \ref{main:su3}. 
  In particular,   $\Phi_{s_a}(\mS_0(X))\subset \mS_c(X)$.
  
 For  $f\in \mS_0(X)$ let us show that $\Phi_s(f)\in \mS_0(X)$.
 For any triple $(L_\alpha, \phi_a\circ t,\chi)\in \bB_a$ denote
 the Mellin transform by $P(\chi,\alpha).$
 Then $P(\chi,s_a(\alpha))\Phi_{s_a}(f)=\Phi_{s_a}(P(\chi,\alpha)f)=0$
 by the equivariance property of $\Phi_{s_a}$. 
  
\end{proof}

\subsubsection{The operator $\kappa_\Psi(\Phi_{s_a})$}

In this subsection we compute $\kappa_\Psi(\Phi_{s_a})$
for the character $\Psi$ defined in \ref{subsection:general:whitt}.

\begin{Prop}\label{kappa:Phi:s}
  For any $\varphi\in \mS_0(T)$ one has 
$$\kappa_{\Psi}(\Phi_{s_a})(\varphi)= s_a\cdot \varphi.$$
\end{Prop}
\begin{proof}
We shall show first the statement for $t=1$, i.e.
$$\kappa_\Psi(\Phi_{s_a})(\varphi)(1)=\theta(t_a)\varphi(1).$$

Let $a$ be a positive root ray. 

$$\mW_\Psi(\Phi_{s_a}(f))(1)=
\int\limits_{U^{op}} \Phi_{s_a}(f)([u])\Psi^{-1}(u) du.$$

We use decomposition $U^{op}=U^{-a} U_{-a}$ where $U^{-a}$ is the product
of all root subgroups corresponding to the negative root rays, except $-a$.
One has
$$\mW_\Psi(\Phi_{s_a}(f))(1)=\int\limits_{U^{-a}} \left(\int\limits_{U_{-a}}
\Phi_{s_a}(f)([u_1 u_2])\Psi^{-1}(u_1)du_1\right) \cdot \Psi^{-1}(u_2) du_2.$$
The character $\Psi$ restricted to $U_{-a}$ equals $\Psi^a_1$
by the definition of $\Psi$.
The inner integral equals
$$\int\limits_{U_{-a}}
\Phi_{s_a}(\iota_a(f)(u_2))([u_1] ) \Psi_1^a(u_1^{-1})du_1=
\mW_{\Psi^a_1}(\Phi_{s_a}(\iota_a(f)(u_2))(1).$$
By Theorems \ref{main:sl2} and \ref{main:su3} it equals to 
$ \mW_{\Psi^a_1}(\theta(t_a)\iota_a(f)(u_2))(1).$

Thus
$\mW_\Psi(\Phi_{s_a}(f))(1)$ equals
$$\int\limits_{U^{-a}} \left(\int\limits_{U_{-a}}
(\theta(t_a)\iota_a(f))(u_2)([u_3])
\Psi^{-1}(u_3)du_3\right) \cdot \Psi^{-1}(u_2) du_2=
\mW_\Psi(\theta(t_a)f)(1).$$

For an arbitrary $t\in T$ one has 
$$\kappa_\Psi(\Phi_{s_a})(\varphi)(t)=
\kappa_\Psi(\theta(t) \Phi_{s_a})(\varphi)(1)=$$
$$\kappa_\Psi(\Phi_{s_a})(\theta(t^{s_a})\varphi)(1)=
\theta(t_a t^{s_a})\varphi(1)=$$
$$\varphi(t_at^{s_a})={s_a}\cdot \varphi(t)$$
as required. 
\end{proof}

Now we are ready to prove Theorem \ref{main}. 
\begin{proof}
  The injectivity of $\kappa_\Psi$ implies the uniqueness of the
  family $\Phi_w, w\in W$. To prove  Theorem  it is enough to construct
  the operators $\Phi_w$.
 For any $w\in W$ there is a  presentation
 $w=s_{a_1}\cdot \ldots \cdot s_{a_n}$  as a product of simple reflections.
  We define the operator
  $\Phi_w\in \Aut(L^2(X))$
  $$\Phi_w(f)=\Phi_{s_{a_1}}\circ \ldots \circ\Phi_{s_{a_n}}.$$
 
  The operator $\Phi_w$ is unitary, preserves $\mS_0(X)$  and satisfies
  $\theta(g,t)\circ \Phi_w=\Phi_w\circ \theta(g,t^w)$ for
  $g\in G, t\in T$.

  Clearly, $\kappa_\Psi$ is a homomorphism of algebras.
  In particular,
$$\kappa_\Psi(\Phi_w)(\varphi)=\kappa_\Psi(\Phi_{s_{a_1}})\circ\ldots \circ
\kappa_\Psi(\Phi_{s_{a_n}})(\varphi)=
s_{a_1}\cdot \ldots s_{a_n}\cdot \varphi=w\cdot \varphi,$$
and hence $\kappa_\Psi(\Phi_w)$ does not depend on the presentation
of $w$. Since $\kappa_\Psi$ is injective, the operator $\Phi_w$
does not depend on the presentation of $w$. The property
$\Phi_{w_1w_2}=\Phi_{w_1}\circ \Phi_{w_2}$ is obvious from the definition. 
\end{proof}

\bibliographystyle{alpha}
 \bibliography{bib}

\end{document}